\documentclass[11pt]{article}
\usepackage{amsfonts}
\usepackage{amssymb}
\usepackage{hyperref}
\usepackage{epsfig}
\usepackage{graphicx}
\usepackage{amsmath}
\usepackage{amsthm}

\newtheorem{theorem}{Theorem}[section]
\newtheorem{definition}[theorem]{Definition}
\newtheorem{proposition}[theorem]{Proposition}
\newtheorem{lemma}[theorem]{Lemma}
\newtheorem{claim}[theorem]{Claim}
\newtheorem{corollary}[theorem]{Corollary}
\newtheorem{remark}[theorem]{Remark}

\newtheorem{question}[theorem]{Question}

\newenvironment{example}{\noindent{\bf Example}\hspace*{1em}}{\bigskip}

\newcommand{\remove}[1]{}

\newcommand{\comment}[1]{}

\def\phi{\varphi}

\newcommand{\mathify}[1]{\ifmmode{#1}\else\mbox{$#1$}\fi}

\renewenvironment{proof}[1][Proof]{
\noindent \textbf{#1: }}{$\Box$
\bigskip}

\newcommand{\ignore}[1]{}

\renewcommand{\Pr}{{\bf Pr}}

\begin{document}

\title{On The Probability of a Rational Outcome for Generalized
Social Welfare Functions on Three Alternatives}

\author{
Nathan Keller\footnote{This research is supported by the Adams
Fellowship Program of the Israel Academy of Sciences and Humanities.} \\
Einstein Institute of Mathematics, Hebrew University\\
Jerusalem 91904, Israel\\
{\tt nkeller@math.huji.ac.il}\\
}

\maketitle

\begin{abstract}
In~\cite{Kalai-Choice}, Kalai investigated the probability of a
rational outcome for a generalized social welfare function (GSWF)
on three alternatives, when the individual preferences are uniform
and independent. In this paper we generalize Kalai's results to a
broader class of distributions of the individual preferences, and
obtain new lower bounds on the probability of a rational
outcome in several classes of GSWFs. In particular, we show that
if the GSWF is monotone and balanced and the distribution of the
preferences is uniform, then the probability of a rational outcome
is at least 3/4, proving a conjecture raised by Kalai. The tools
used in the paper are analytic: the Fourier-Walsh expansion of
Boolean functions on the discrete cube, properties of the
Bonamie-Beckner noise operator, and the FKG inequality.

\end{abstract}

\section{Introduction}

Consider a situation in which a society of $n$ members selects a
ranking amongst $m$ alternatives. In the election process, each
member of the society gives a ranking of the alternatives (the
ranking is a full linear ordering; that is, indifference between
alternatives is not allowed). The set of the rankings given by the
individual members is called a {\it profile}. Given the profile,
the ranking of the society is determined according to some
function, called a {\it generalized social welfare function}
(GSWF).

The GSWF is a function $F:L^n \rightarrow
\{0,1\}^{{{m}\choose{2}}}$, where $L$ is the set of linear
orderings on $m$ elements. In other words, given the profile
consisting of linear orderings supplied by the voters, the
function determines the preference of the society amongst each of
the ${{m}\choose{2}}$ pairs of alternatives. If the output of $F$
can be represented as a full linear ordering of the $m$
alternatives, then $F$ is called a {\it social welfare function}
(SWF).

Throughout this paper we consider GSWFs satisfying the
{\it Independence of Irrelevant Alternatives} (IIA) condition: For every
two alternatives $A$ and $B$, the preference of the entire society
between $A$ and $B$ depends only on the preference of each
individual voter between $A$ and $B$. This natural condition on
GSWFs can be traced back to Condorcet~\cite{Condorcet}.

The Condorcet's paradox demonstrates that if the number of
alternatives is at least three and the GSWF is based on the
majority rule between every pair of alternatives, then there exist
profiles for which the voting procedure cannot yield a full order
relation. That is, there exist alternatives $A,B,$ and $C$, such
that the majority of the society prefers $A$ over $B$, the
majority prefers $B$ over $C$, and the majority prefers $C$ over
$A$. Such situation is called {\it irrational choice} of the
society. Arrow's impossibility theorem~\cite{Arrow} asserts that
if a GSWF on at least three alternatives satisfies the IIA
condition, has all the possible orderings of the alternatives in
its range, and is not a dictatorship (that is, the preference of
the society is not determined by a single member), then there
exists a profile for which the choice of the society is irrational.

Since the existence of profiles leading to an irrational choice
has significant implications on voting procedures, an extensive
research has been conducted in order to evaluate the probability
of irrational choice for various GSWFs. Most of the results in
this area are summarized in~\cite{Gehrlein}. In addition to its
significance in Social Choice theory, this area of research leads
to interesting questions in probabilistic and extremal
combinatorics (see~\cite{Mossel}).

In 2002, Kalai~\cite{Kalai-Choice} suggested an analytic approach
to this study. He showed that for GSWFs on three alternatives
satisfying the IIA condition, the probability of irrational choice
can be computed by a formula related to the Fourier-Walsh
expansion of the GSWF. Using this formula he presented a new proof
of Arrow's impossibility theorem under additional assumption of
neutrality and established upper bounds on the probability of
irrational choice for specific classes of GSWFs.

\medskip

In this paper we generalize the results of~\cite{Kalai-Choice} in
several directions. As in~\cite{Kalai-Choice}, we focus on GSWFs
on three alternatives satisfying the IIA condition. We denote the
alternatives by $A,B,$ and $C$, and the choice functions amongst
the pairs $(A,B),(B,C),$ and $(C,A)$ by $f,g,$ and $h$,
respectively (see Figure~\ref{fig:Tri}).


\begin{figure}[tb]
\begin{center}
\scalebox{0.8}{
\includegraphics{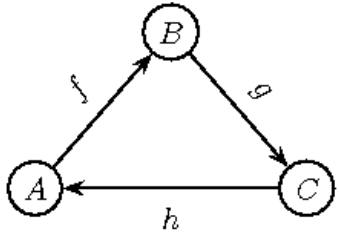}
}
\caption{The Alternatives and the Choice Functions}
\label{fig:Tri}
\end{center}
\end{figure}

We examine GSWFs satisfying (some of) the following
conditions:
\begin{itemize}

\item {\it Balance} - A GSWF is {\it balanced} if the choice
functions $f,g,$ and $h$ are balanced (i.e., satisfy
$\mathbb{E}[f] = \mathbb{E}[g] = \mathbb{E}[h] = 1/2$).

\item {\it Neutrality} - A GSWF is {\it neutral} if it is invariant
under permutations of the alternatives. In particular, this implies
that the choice functions satisfy $f=g=h$, and that $f$ is balanced.

\item {\it Symmetry} - We call a GSWF {\it symmetric} if it is
invariant under a transitive group of permutations of the voters.
In particular, this implies that the choice functions are far from a
dictatorship.\footnote{Note that this definition of symmetry is much
weaker than the usual definition requiring that the function depends
only on the Hamming weight of the input. Important classes of
functions, including the tribes functions~\cite{Ben-Or}, satisfy
our definition of symmetry.}

\item {\it Monotonicity} - A GSWF is {\it monotone} if the choice
functions $f,g,$ and $h$ are monotone increasing.\footnote{The
definition of a monotone increasing function on the discrete cube
is given is Section~\ref{Sec:Lower-Bounds}.}
\end{itemize}

The first direction in our paper is a generalization of the
possible distributions of the individual preferences.
In~\cite{Kalai-Choice} it is assumed that
the individual preferences are independent and uniformly
distributed. We show that the results of~\cite{Kalai-Choice} are
valid (under some modifications) also for non-uniform
distributions of the preferences, as long as the voters are
independent, and for each ordering of the alternatives, the
probability of the ordering is equal to the probability of the
inverse ordering. We call such distributions {\it even product
distributions}. In particular, we prove the following
generalization of Theorem~5.1 of~\cite{Kalai-Choice}:
\begin{theorem}
Consider a GSWF on three alternatives satisfying the IIA
condition. If the distribution of the preferences is an even
product distribution such that the probability of each preference
is positive, and the GSWF is neutral and symmetric, then the
probability of irrational choice is bounded away from zero,
independently of the number of the voters.\footnote{In the context
of this theorem, ``bounded away'' means that the probability is
greater than a constant, depending only on the distribution of
the preferences, and not on the number of voters and the choice
functions. Theorem~5.1 in~\cite{Kalai-Choice} states that if the
preferences are distributed uniformly, then the value of this
constant is at least $0.0808$.}
\label{Thm:Bounded-Away}
\end{theorem}

The second direction is obtaining new lower bounds on the
probability of a rational choice for several classes of GSWFs. In
particular, we prove the following conjecture raised
in~\cite{Kalai-Choice}:
\begin{theorem}
Consider a GSWF on three alternatives satisfying the IIA
condition. If the individual preferences are independent and
uniformly distributed, and the GSWF is monotone and balanced, then
the probability of a rational choice is at least 3/4.
\label{Thm:Monotone}
\end{theorem}
The proof of this result relies on properties of the
Bonamie-Beckner noise operator and uses the FKG inequality~\cite{FKG}.
Furthermore, we establish a generalization of
Theorem~\ref{Thm:Monotone} to even product distributions of the
individual preferences.

\medskip

Finally, we consider the {\it stability version} of Arrow's
theorem presented in~\cite{Kalai-Choice}. This version asserts
that if a balanced GSWF on three alternatives satisfies the IIA
condition and is at least $\epsilon$-far from being a
dictatorship, then it leads to irrational choice with probability
at least $C \cdot \epsilon$, for a universal constant $C$. Kalai
asked whether his proof technique can be extended to an analytic
proof of Arrow's theorem without the neutrality assumption, or
even to a stability version of Arrow's theorem. (Such version
would assert that for any $\epsilon>0$, there exists
$\delta=\delta(\epsilon)$ such that if a GSWF on at least three
candidates satisfies the IIA condition and is at least
$\delta$-far from being a dictatorship and from not having all the
orderings of the alternatives in its range, then the probability
of irrational choice is at least $\epsilon$.)

We show that the neutrality assumption cannot be dropped completely
from Kalai's result, that is, there does not exist a stability version
of Arrow's theorem (with no additional assumptions) in which the
dependence of $\delta(\epsilon)$ on $\epsilon$ is linear.
\begin{theorem}
For all $\epsilon,K > 0$ and $n=n(\epsilon,K)$ big enough, there
exists a GSWF on three alternatives satisfying the IIA condition,
such that:
\begin{enumerate}
\item Amongst any pair of alternatives, the probability of each
alternative to be preferred by the society over the other
alternative is at least $\eta=2^{-\epsilon n}/(n+1)$.

\item The probability of an irrational choice is less than $\eta/K$.
\end{enumerate}
\label{Thm:Instability}
\end{theorem}
The example that proves Theorem~\ref{Thm:Instability} is a GSWF on
three alternatives in which the choice functions $f,g,$ and $h$
are threshold functions (i.e., $(f(x)=1) \Leftrightarrow \left(
\sum_{i=1}^n x_i \geq k \right)$ ), with expectations $\eta,1/2,$
and $1-\eta$.

\medskip

After this paper was written, a stability version of Arrow's theorem
without additional assumptions was proved by Mossel~\cite{Mossel-Arrow}.
In Mossel's theorem, the dependence of $\delta$ on $\epsilon$ is
$\delta=m^2 \cdot \exp(C/\epsilon^{21})$ for a universal constant $C$, where
$m$ is the number of alternatives. Recently,
Keller~\cite{Keller-new} showed that the stability version holds for
$\delta \approx Cm^2 \cdot \epsilon^{8/9}$ where $C$ is a universal constant.
Moreover, Keller showed that for small values of $\epsilon$, the example
presented above (i.e., the threshold functions) is almost optimal: its
probability of irrational choice is greater than the lower bound at
most by a logarithmic factor (in $\epsilon$).

\medskip

The paper is organized as follows: In Section~\ref{Sec:Fourier} we
recall some basic properties of the Fourier-Walsh expansion of
functions on the discrete cube and of the Bonamie-Beckner
noise operator. In Section~\ref{Sec:Generalization} we generalize the results
of~\cite{Kalai-Choice} to even product distributions of the
preferences and prove Theorem~\ref{Thm:Bounded-Away}. In
Section~\ref{Sec:Lower-Bounds} we establish lower bounds on the
probability of a rational choice for several classes of GSWFs and
prove Theorem~\ref{Thm:Monotone}. In
Section~\ref{Sec:Upper-Bounds} we discuss Kalai's stability version
of Arrow's theorem and prove Theorem~\ref{Thm:Instability}.

\section{Preliminaries} \label{Sec:Fourier}

\subsection{Fourier-Walsh Expansion of Functions on the Discrete Cube}

Consider the discrete cube $\{0,1\}^n$ endowed with the uniform measure
$\mu$. Denote the set of all real-valued functions on the discrete cube
by $X$. The inner product of functions $f,g \in X$ is defined as usual as
\[
\langle f,g \rangle = \mathbb{E}_{\mu}[fg]= \frac{1}{2^n} \sum_{x \in
\{0,1\}^n} f(x)g(x).
\]
This inner product induces a norm on $X$:
\[
||f||_2 = \sqrt{\langle f,f \rangle} = \sqrt{\mathbb{E}_{\mu}[f^2]}.
\]
Consider the Rademacher functions $\{r_i\}_{i=1}^n$, defined as:
\[
r_i(x_1,\ldots,x_n)=2x_i-1.
\]
These functions constitute an orthonormal system in $X$. Moreover, this
system can be completed to an orthonormal basis in $X$ by defining
\[
r_S = \prod_{i \in S} r_i,
\]
for all $S \subset \{1,\ldots,n\}$. Every function $f \in X$ can be
represented by its Fourier expansion with respect to the system
$\{r_S\}_{S \subset \{1,\ldots,n\}}$:
\[
f = \sum_{S \subset \{1,\ldots,n\}} \langle f,r_S \rangle r_S.
\]
This representation is called the Fourier-Walsh expansion of $f$. The
coefficients in this expansion are denoted by
\[
\hat f(S) = \langle f,r_S \rangle,
\]
and the {\it level} of the coefficient $\hat f(S)$ is $|S|$.

\medskip

\noindent By the Parseval identity, for all $f \in X$,
\[
\sum_{S \subset \{1,\ldots,n\}} \hat f(S)^2 = ||f||_2^2.
\]
More generally, for all $f,g \in X$,
\[
\langle f,g \rangle = \sum_{S \subset \{1,\ldots,n\}} \hat f(S)
\hat g(S).
\]
Following~\cite{Kalai-Choice}, we will be also interested in a biased
version of the inner product, defined as follows:
\begin{definition}
Let $f,g$ be two real-valued functions on the discrete cube, and let $-1
\leq \delta \leq 1$. Define
\[
\langle\langle f,g \rangle\rangle_{\delta} = \sum_{S \neq
\emptyset} \hat f(S) \hat g(S) \delta^{|S|}.
\]
\end{definition}
\noindent Note that this definition slightly differs from the
definition used in~\cite{Kalai-Choice}. Finally, we note that for
all $f \in X$,
\[
\hat f(\emptyset) = \langle f, r_{\emptyset} \rangle = \mathbb{E}_{\mu} [f \cdot 1]
= \mathbb{E}_{\mu}[f].
\]

\subsection{The Bonamie-Beckner Noise Operator}
\label{sec:sub:noise}

The {\it noise operator}, introduced in~\cite{Bonamie,Beckner}, is defined in terms
of the Fourier-Walsh expansion as follows:
\begin{definition}
Consider a function $f$ on the discrete cube with a Fourier-Walsh
expansion $f=\sum_S \hat f(S) r_S$. For $0 \leq \epsilon \leq
1$, the noise operator $T_{\epsilon}$ applied to $f$ is
\begin{equation}
T_{\epsilon} f = \sum_S \epsilon^{|S|} \hat f(S) r_S.
\label{Eq:Beckner0}
\end{equation}
\end{definition}
\noindent It is well-known that one can arrive from $f$ to $T_{\epsilon} f$
by the following process: For any $x \in \{0,1\}^n$,
\begin{equation}
T_{\epsilon} f(x) = \mathbb{E} [f(x \oplus y)],
\label{Eq:Beckner1}
\end{equation}
where $\oplus$ denotes coordinate-wise addition modulo $2$, and each
coordinate of $y$ is chosen independently according to the
distribution $\Pr[y_i=0]=(1+\epsilon)/2, \Pr[y_i=1]=(1-\epsilon)/2$.
That is, each coordinate of $x$ is left unchanged with probability
$\epsilon$ and is replaced by a random value with probability $1-\epsilon$,
and then $f$ is evaluated on the result. Thus, $T_{\epsilon} f$ represents
a noisy variant of $f$, and for this reason $T_{\epsilon} f$ is
called ``the noise operator''.

\medskip

\noindent As pointed out by the anonymous
referee, the noise operator can be defined in the same way
(i.e., by Equation~\ref{Eq:Beckner0}) also for $-1 \leq \epsilon \leq 0$.
Moreover, it can be easily shown that the basic property of the noise
operator described above (i.e., Equation~\ref{Eq:Beckner1}) also translates
to the case $-1 \leq \epsilon \leq 0$. That is, we still have
\[
T_{\epsilon} f(x) = \mathbb{E} [f(x \oplus y)],
\]
where each coordinate of $y$ is chosen independently according to
the distribution $\Pr[y_i=0]=(1+\epsilon)/2,
\Pr[y_i=1]=(1-\epsilon)/2$. Using this observation, we shall
consider the noise operator for $-1 \leq \epsilon \leq 1$.

\section{The Probability of Rational Choice for a Non-Uniform Distribution
of the Preferences} \label{Sec:Generalization}

Throughout the paper we assume that the number of alternatives is
three and denote the alternatives by $A,B,$ and $C$. Since (by
assumption) the GSWF satisfies the IIA condition, the preference
of the society between every pair of alternatives can be
represented by a Boolean function on the discrete cube. Formally,
given a profile, we consider the pair of alternatives $(A,B)$ and
construct a binary vector $(x_1,\ldots,x_n)$ such that $x_i=1$ if
the $i$-th voter prefers $A$ over $B$, and $x_i=0$ if the $i$-th
voter prefers $B$ over $A$. We set $f(x_1,\ldots,x_n)=1$ if the
entire society prefers $A$ over $B$ and $f(x_1,\ldots,x_n)=0$ if
the society prefers $B$ over $A$. Note that the preference of the
society between $A$ and $B$ is determined by $(x_1,\ldots,x_n)$,
and hence $f$ is well-defined. Similarly, we define the Boolean
functions $g$ and $h$ that represent the preferences between the
pairs $(B,C)$ and $(C,A)$, respectively (see Figure~\ref{fig:Tri}).

Every profile is uniquely represented by the binary vector
$(x_1,\ldots,x_n,y_1,\ldots,y_n,z_1,\ldots,z_n)$, where
$(x_i,y_i,z_i)$ represent the preferences of the $i$-th voter
between $(A,B),(B,C)$, and $(C,A)$. We assume that the vectors
$(x_i,y_i,z_i)$ for different values of $i$ are independent (i.e.,
the preferences of the individual voters are independent), and
that these vectors do not assume the values $(0,0,0)$ and
$(1,1,1)$ (since otherwise the preferences of the $i$-th voter do
not constitute an order relation). In~\cite{Kalai-Choice}, the
distribution over the six possible values of $(x_i,y_i,z_i)$ was
assumed to be uniform. In our analysis, we consider the following
distribution:
\[
\Pr[(x_i,y_i,z_i)=(1,1,0)]=\alpha \qquad
\Pr[(x_i,y_i,z_i)=(0,1,1)]=\beta \qquad
\Pr[(x_i,y_i,z_i)=(1,0,1)]=\gamma
\]
\[
\Pr[(x_i,y_i,z_i)=(0,0,1)]=\alpha \qquad
\Pr[(x_i,y_i,z_i)=(1,0,0)]=\beta \qquad
\Pr[(x_i,y_i,z_i)=(0,1,0)]=\gamma,
\]
where $\alpha+\beta+\gamma=1/2$. We call this distribution an {\it
even product distribution}, and denote it by
$D(\alpha,\beta,\gamma)$. The intuition behind the restrictions
will be explained at the end of this section.

\begin{theorem}
Consider a GSWF on three alternatives satisfying the IIA condition
where the choice functions between the pairs of alternatives
$(A,B),(B,C),$ and $(C,A)$ are $f,g,$ and $h$, respectively. If
the distribution of the individual preferences is an even product
distribution $D(\alpha,\beta,\gamma)$, as described above, then the
probability of irrational choice is given by the formula:
\begin{equation}
W(f,g,h)=p_1 p_2 p_3 + (1-p_1)(1-p_2)(1-p_3) + \langle\langle f,g
\rangle\rangle_{4 \alpha -1} + \langle\langle g,h
\rangle\rangle_{4 \beta -1} + \langle\langle h,f \rangle\rangle_{4
\gamma -1},
\label{Eq3.0}
\end{equation}
where $p_1,p_2,$ and $p_3$ are the expectations of $f,g,$ and $h$,
respectively.
\label{Thm:Formula}
\end{theorem}
\begin{remark} Theorem~\ref{Thm:Formula} generalizes Theorem~3.1
of~\cite{Kalai-Choice}, which corresponds to the case
$\alpha=\beta=\gamma=1/6$.
\end{remark}

\begin{proof}
For a profile
$(x,y,z)=(x_1,\ldots,x_n,y_1,\ldots,y_n,z_1,\ldots,z_n)$, the
choice of the society is rational if and only if
\[
f(x)g(y)h(z)+(1-f(x))(1-g(y))(1-h(z))=0.
\]
Therefore, the probability of irrational choice is
\[
W(f,g,h)=\sum_{(x,y,z) \in \{0,1\}^{3n}} \Pr[(x,y,z)]
\Big(f(x)g(y)h(z)+(1-f(x))(1-g(y))(1-h(z))\Big),
\]
where $\Pr[(x,y,z)]=\prod_i \Pr[(x_i,y_i,z_i)]$, according to the
distribution $D(\alpha,\beta,\gamma)$.

\medskip

\noindent Consider the functions $F_1,F_2,F_3:\{0,1\}^{3n} \rightarrow
\mathbb{R}$ defined by
\[
F_1(x_1,\ldots,x_n,y_1,\ldots,y_n,z_1,\ldots,z_n)=f(x)g(y)h(z),
\]
\[
F_2(x_1,\ldots,x_n,y_1,\ldots,y_n,z_1,\ldots,z_n)=(1-f(x))(1-g(y))(1-h(z)),
\]
\[
F_3(x_1,\ldots,x_n,y_1,\ldots,y_n,z_1,\ldots,z_n)=\Pr[(x,y,z)].
\]
We have
\[
W(f,g,h)=2^{3n} \langle F_3,F_1+F_2 \rangle,
\]
and hence by the Parseval identity,
\begin{equation}
W(f,g,h) = 2^{3n} \sum_{S \subset \{1,\ldots,3n\}} \hat F_3(S)
(\hat F_1(S) + \hat F_2(S)).
\label{Eq3.1}
\end{equation}
Therefore, in order to compute the probability of rational choice
it is sufficient to compute the Fourier-Walsh expansions of
$F_1,F_2,$ and $F_3$.

In order to compute the expansions, we use the fact that if a
function is a multiplication of functions on disjoint sets of
variables, then its Fourier-Walsh expansion also has the same
structure. Hence, if we denote $S=(S_1,S_2,S_3)$, where $S_1$
represents $(x_1,\ldots,x_n)$, $S_2$ represents
$(y_1,\ldots,y_n)$, and $S_3$ represents $(z_1,\ldots,z_n)$, then
\[
\hat F_1(S) = \hat f(S_1) \hat g(S_2) \hat h(S_3) \qquad
\mbox{ and } \qquad
\hat F_2(S) = \widehat{1-f}(S_1) \widehat{1-g}(S_2) \widehat{1-h}(S_3).
\]
Similarly, since the individual preferences are independent, the
Fourier-Walsh expansion of $F_3$ is determined by the
Fourier-Walsh expansion of the functions $F_4^i: \{0,1\}^3
\rightarrow \mathbb{R}$ defined by
\[
F_4^i((x_i,y_i,z_i)) = \Pr[(x_i,y_i,z_i)].
\]
This expansion (presented below) can be found by direct
computation.
\[
\hat F_4^i(\emptyset)=1/8, \qquad \hat F_4^i(\{1\})=0, \qquad \hat
F_4^i(\{2\})=0, \qquad \hat F_4^i(\{3\})=0, \qquad \hat
F_4^i(\{1,2\})=(4 \alpha -1)/8,
\]
\[
\hat F_4^i(\{2,3\})=(4 \beta -1)/8, \qquad \hat F_4^i(\{1,3\})=(4
\gamma -1)/8, \qquad \hat F_4^i(\{1,2,3\})=0.
\]
Since the Fourier-Walsh coefficients of $F_3$ are multiplications
of the corresponding coefficients of $\{F_4^i\}_{i=1}^n$, we have
$\hat F_3(S)=0$, unless $S=(S_1,S_2,S_3)$ has a special structure:
Each $1 \leq i \leq n$ is contained in either none or two of the
sets $(S_1,S_2,S_3)$. For such special sets $S$, the coefficients
are given by the formula
\[
\hat F_3(S) = \Big(\frac{1}{8}\Big)^{t_1} \Big(\frac{4 \alpha -1}{8}\Big)^{t_2}
\Big(\frac{4 \beta -1}{8}\Big)^{t_3} \Big(\frac{4 \gamma -1}{8}\Big)^{t_4},
\]
where
\[
t_1= \mbox{ the number of triples } (x_i,y_i,z_i) \mbox{ equal to
} (0,0,0),
\]
\[
t_2= \mbox{ the number of triples } (x_i,y_i,z_i) \mbox{ equal to
} (1,1,0),
\]
\[
t_3= \mbox{ the number of triples } (x_i,y_i,z_i) \mbox{ equal to
} (0,1,1),
\]
\[
t_4= \mbox{ the number of triples } (x_i,y_i,z_i) \mbox{ equal to
} (1,0,1).
\]

Finally, we note that by the linearity of the Fourier transform,
we have $\hat f(S_1) = -(\widehat{1-f}(S_1))$ for all $S_1 \neq \emptyset$,
and the same for $g$ and $h$. Therefore, if $S_1,S_2,S_3 \neq \emptyset$, then
\[
\hat F_1(S) + \hat F_2(S) = 0.
\]
Combining the observations above, we get that the term
\[
\hat F_3(S) (\hat F_1(S) + \hat F_2(S))
\]
vanishes unless $S=(S_1,S_2,S_3)$ has the following special
structure: At least one of $S_1,S_2,S_3$ is empty, and each $i$ is
contained in either none or two of $S_1,S_2,S_3$.

Assume that $S_3=\emptyset$, and thus $S_1=S_2$ (otherwise, there exists
$i$ that is contained in only one of the sets $S_1,S_2,S_3$, and hence $\hat
F_3(S) (\hat F_1(S) + \hat F_2(S))=0$). Assume also that $S_1 \neq
\emptyset$. We note that $\widehat{1-h}(\emptyset)=1- \hat
h({\emptyset})$, and hence by the calculations above,
\[
\hat F_3(S) (\hat F_1(S) + \hat F_2(S)) = \Big(\frac{1}{8}\Big)^{n-|S_1|}
\Big(\frac{4 \alpha -1}{8}\Big)^{|S_1|} \hat f(S_1) \hat g(S_1) =
\Big(\frac{1}{8}\Big)^n (4 \alpha -1)^{|S_1|} \hat f(S_1) \hat g(S_1).
\]
If $S_1=S_2=S_3=\emptyset$, then
\[
\hat F_3(S) (\hat F_1(S) + \hat F_2(S)) = (1/8)^n (p_1 p_2 p_3 +
(1-p_1)(1-p_2)(1-p_3)).
\]
Therefore, summing over all the possible values of $S$ we get
\[
\sum_{S \subset \{1,\ldots,3n\}} \hat F_3(S) (\hat F_1(S) + \hat
F_2(S)) = (1/8)^n \Big(p_1 p_2 p_3 + (1-p_1)(1-p_2)(1-p_3) +
\]
\[
+ \sum_{S_1 \neq \emptyset} (4 \alpha -1)^{|S_1|} \hat f(S_1) \hat
g(S_1) + \sum_{S_2 \neq \emptyset} (4 \beta -1)^{|S_2|} \hat
g(S_2) \hat h(S_2) + \sum_{S_3 \neq \emptyset} (4 \gamma
-1)^{|S_3|} \hat f(S_3) \hat h(S_3) \Big) =
\]
\[
= (1/8)^n \Big(p_1 p_2 p_3 + (1-p_1)(1-p_2)(1-p_3)+ \langle\langle f,g
\rangle\rangle_{4 \alpha -1} + \langle\langle g,h
\rangle\rangle_{4 \beta -1} + \langle\langle h,f \rangle\rangle_{4
\gamma -1} \Big),
\]
and thus the assertion of the theorem follows from
Equation~(\ref{Eq3.1}).
\end{proof}

Using Theorem~\ref{Thm:Formula}, some of the results
of~\cite{Kalai-Choice} and~\cite{Mossel} can be generalized to even product
distributions of the preferences. We present here two of the
results.

\begin{theorem}
Consider a GSWF on three alternatives satisfying the IIA
condition. If the distribution of the preferences is an even product
distribution $D(\alpha,\beta,\gamma)$ and the GSWF is neutral and
symmetric, then the probability of an irrational choice
satisfies the inequality
\begin{equation}
W(f,g,h) \geq (\frac{1}{4}-d_m)(1+(4 \alpha -1)^3 + (4 \beta -1)^3
+ (4 \gamma -1)^3) > 0,
\label{Eq3.2}
\end{equation}
where $d_m \approx 1/(2\pi)$ is the sum of squares of the first-level
Fourier-Walsh coefficients of the majority function. In particular,
$W(f,g,h)$ is bounded away from zero.
\label{Cor:Neutral-Symmetric}
\end{theorem}

\begin{remark} Theorem~\ref{Cor:Neutral-Symmetric} generalizes
Theorem~5.1 in~\cite{Kalai-Choice}, which corresponds to the case
$\alpha=\beta=\gamma=1/6$.
\end{remark}

\noindent In the proof of Theorem~\ref{Cor:Neutral-Symmetric} we use the following technical
lemma, obtained with the assistance of Tomer Schlank.
\begin{lemma}\label{Lemma0.0}
For any integer $k \geq 1$, and all $-1 \leq x,y,z \leq 1$ such
that $x+y+z=1$, we have
\begin{equation}\label{Eq0}
x^3+y^3+z^3 \geq x^{2k+1} + y^{2k+1} + z^{2k+1}.
\end{equation}
\end{lemma}

\begin{proof} Denote $D= \{(x,y,z) \in [-1,1]^3 |
x+y+z=1\}$, and $f(x,y,z)= (x^3+y^3+z^3) - (x^{2k+1} + y^{2k+1} +
z^{2k+1})$. Since $D$ is compact and $f$ is continuous, $f$
obtains a minimum in $D$. We would like to show that $\min_D (f) =
0$. First, we note that $f$ is identically zero on the boundary of
$D$. Indeed, if $(x,y,z) \in \partial(D)$, then w.l.o.g., either
$x=-1$ and then necessarily $y=z=1$, or $x=1$ and then $y=-z$. In
both cases, $f(x,y,z)=0$. If $f$ attains its minimum in an
internal point $(x_0,y_0,z_0) \in D$, then by Lagrange
multipliers, we have
\[
3x_0^2 - (2k+1)x_0^{2k} = 3y_0^2 - (2k+1)y_0^{2k} = 3z_0^2 -
(2k+1)z_0^{2k}.
\]
If $|x_0| \neq |y_0|$, the first equality is equivalent to:
\begin{equation}\label{Eq1}
\frac{3}{2k+1} = \frac{x_0^{2k}-y_0^{2k}}{x_0^2-y_0^2} =
\sum_{l=0}^{k-1} (x_0^2)^l (y_0^2)^{k-1-l},
\end{equation}
and similarly for the pairs $(x_0,z_0)$ and $(y_0,z_0)$. For a
given $x_0$, the function $\sum_{l=0}^{k-1} (x_0^2)^l
(y_0^2)^{k-1-l}$ is increasing as function of $y_0^2$. Hence,
Equation~(\ref{Eq1}) can be satisfied for both $(x_0,y_0)$ and
$(x_0,z_0)$ only if $|y_0|=|z_0|$. Thus, an internal minimum point
of $f$ in $D$ must satisfy at least one of the conditions
$|x_0|=|y_0|$, $|x_0|=|z_0|$ or $|y_0|=|z_0|$. Assume, w.l.o.g.,
that $|x_0|=|y_0|$. If $x_0=-y_0$, then necessarily $z_0=1$, and
thus $(x_0,y_0,z_0) \in \partial (D)$. If $x_0=y_0$, then
$z_0=1-2x_0$, and hence, Inequality~(\ref{Eq0}) is reduced to:
\begin{equation}\label{Eq2}
2x_0^3+(1-2x_0)^3 \geq 2x_0^{2k+1} + (1-2x_0)^{2k+1}.
\end{equation}
Therefore, it is sufficient to prove Inequality~(\ref{Eq2}) for
all $0 \leq x_0 \leq 1$. Note that the inequality holds trivially
for $x_0 \leq 1/2$. Let $g(t)=t^3-t^{2k+1}$, and denote
$\delta=1-x$. By Inequality~(\ref{Eq2}), it is sufficient to prove
that for all $0 \leq \delta \leq 1/2$,
\begin{equation}\label{Eq3}
2g(1-\delta)\geq g(1-2 \delta).
\end{equation}
We use the following two properties of $g(t)$:
\begin{enumerate}
\item $g(t)$ is nonnegative for all $0 \leq t \leq 1$.
Furthermore, $g$ is monotone increasing for $0 <t < t_0$ and
monotone decreasing for $t_0 < t < 1$, where $t_0 =
(\frac{3}{2k+1})^{1/(2k-2)}$.

\item $g(t)$ is convex in the domain $0 <t <t_1$, and concave in
the domain $t_1 < t<1$, where $t_1 =
(\frac{6}{2k(2k+1)})^{1/(2k-2)}$.
\end{enumerate}
Since $g(1)=0$, Inequality~(\ref{Eq3}) follows from the concavity
of $g$ whenever $1-2 \delta \geq t_1$. Furthermore, when $1-
\delta \leq t_0$, the inequality follows immediately from the
monotonicity and nonnegativity of $g$ in that domain. The only
remaining case is when $1-2 \delta < t_1$ and $1- \delta > t_0$
(or equivalently, $(1-t_1)/2 < \delta < 1 -t_0$. We note that this
domain may be empty, and in this case we are already done by the
previous considerations). In this case, by the monotonicity
properties of $g$ we have $g(1-\delta)>g((1+t_1)/2)$ (since $t_0 <
1- \delta < (1+t_1)/2$), and $g(1-2 \delta) < g(t_1)$ (since $1 -2
\delta < t_1 < t_0$). Therefore,
\[
2g(1-\delta) - g(1-2 \delta) > 2g((1+t_1)/2)-g(t_1) \geq 0,
\]
where the last inequality follows from the concavity of $g(t)$ for
$t_1 <t <1$. This completes the proof.
\end{proof}

\medskip

\noindent \textbf{Proof of Theorem~\ref{Cor:Neutral-Symmetric}.}
By assumption, the GSWF is neutral, and hence, balanced.
Therefore, by Theorem~\ref{Thm:Formula}, the probability of
irrational choice in our case is
\[
W(f,g,h)=1/4 + \langle\langle f,g \rangle\rangle_{4 \alpha -1} +
\langle\langle g,h \rangle\rangle_{4 \beta -1} + \langle\langle
h,f \rangle\rangle_{4 \gamma -1}.
\]
Since the GSWF is neutral and symmetric, we have $f=g=h$, and all
the Fourier-Walsh coefficients of $f$ on the even non-zero levels vanish
(see~\cite{Kalai-Choice}, Proof of Theorem~5.1). Thus,
\[
W(f,g,h)=1/4 + \sum_{|S| \mbox{ odd }} \hat f(S)^2 (4
\alpha-1)^{|S|} + \sum_{|S| \mbox{ odd }} \hat f(S)^2 (4
\beta-1)^{|S|} + \sum_{|S| \mbox{ odd }} \hat f(S)^2 (4
\gamma-1)^{|S|} =
\]
\[
= 1/4 + \sum_{k=0}^{\lceil n/2 \rceil -1} \Big[((4 \alpha -1)^{2k+1} +
(4 \beta -1)^{2k+1} + (4 \gamma -1)^{2k+1}) \sum_{|S|=2k+1} \hat
f(S)^2 \Big]=
\]
\[
= 1/4 - \sum_{|S|=1} \hat f(S)^2 + \sum_{k=1}^{\lceil n/2 \rceil
-1} ((4 \alpha -1)^{2k+1} + (4 \beta -1)^{2k+1} + (4 \gamma
-1)^{2k+1}) \sum_{|S|=2k+1} \hat f(S)^2,
\]
where the last equality follows from the relation $\alpha+\beta+\gamma=1/2$.
Since for every $k$ the expression $\sum_{|S|=2k+1} \hat f(S)^2$
is non-negative, and since by Lemma~\ref{Lemma0.0}, for all $k \geq 1$,
\[
(4 \alpha -1)^{2k+1} + (4 \beta -1)^{2k+1} + (4 \gamma -1)^{2k+1}
\geq (4 \alpha -1)^3 + (4 \beta -1)^3 + (4 \gamma -1)^3,
\]
it follows that
\[
W(f,g,h) \geq 1/4 - \sum_{|S|=1} \hat f(S)^2 + \left( (4 \alpha
-1)^{3} + (4 \beta -1)^{3} + (4 \gamma -1)^{3} \right) \sum_{|S|>1} \hat f(S)^2=
\]
\[
= (1/4 - \sum_{|S|=1} \hat f(S)^2) (1 + (4 \alpha
-1)^{3} + (4 \beta -1)^{3} + (4 \gamma -1)^{3}),
\]
where the last equality follows from the Parseval identity. Since amongst the
symmetric neutral functions, the expression $\sum_{|S|=1} \hat f(S)^2$ is maximized
for the majority function (see proof of Theorem~5.1 in~\cite{Kalai-Choice}),
we get
\[
W(f,g,h) \geq (\frac{1}{4}-d_m)(1+(4 \alpha -1)^3 + (4 \beta -1)^3
+ (4 \gamma -1)^3),
\]
and thus it is only left to show that
\begin{equation}
(4 \alpha -1)^{3} + (4 \beta -1)^{3} + (4 \gamma -1)^{3}
> -1.
\label{Eq3.3}
\end{equation}
This claim is trivial for
$\alpha, \beta, \gamma \leq 1/4$, since in that case
\[
(4 \alpha -1)^{3} + (4 \beta -1)^{3} + (4 \gamma -1)^{3}
> (4 \alpha -1) + (4 \beta -1) + (4 \gamma -1) = -1.
\]
Hence, assume that $\gamma> 1/4$, and write $\gamma = 1/2 - \alpha
- \beta$ (and thus $4 \gamma -1 = 1- 4 \alpha - 4 \beta$).
Inequality~(\ref{Eq3.3}) is equivalent to
\[
(1 - 4 \alpha)^{3} + (1-4 \beta)^{3} < 1 + (1 - 4 \alpha - 4
\beta)^{3},
\]
that follows from the strict convexity of the function
$F(t)=t^{3}$ on $[0,1]$. This completes the proof of
Theorem~\ref{Cor:Neutral-Symmetric}. $\square$

\medskip

The second result is a combination of Theorem~\ref{Thm:Formula}
with the following proposition, which is an easy consequence of
the ``Majority is stablest'' theorem~\cite{Mossel}:
\begin{proposition}
Let $0 \leq \rho \leq 1$ and let $\epsilon >0$. There exists
$n_0=n_0(\rho,\epsilon)$ such that for all $n>n_0$, if
$f:\{0,1\}^n \rightarrow [0,1]$ is symmetric and balanced then
\[
\langle\langle f,f \rangle\rangle_{\rho} = \sum_{S \neq \emptyset}
\hat f(S)^2 \rho^{|S|} \leq \frac{1}{2\pi} \arcsin \rho +
\epsilon.
\]
\label{Prop:Maj-Stablest}
\end{proposition}

\begin{corollary}
Consider a GSWF on three alternatives, where the distribution of
the preferences is an even product distribution $D(\alpha,\beta,\gamma)$
with $\alpha,\beta, \gamma \leq 1/4$. Then for all $\epsilon>0$ there exists
$n_0=n_0(\epsilon, \alpha, \beta, \gamma)$ such that if the number
of voters is $n>n_0$ and the GSWF is neutral, symmetric, and
satisfies the IIA condition, then the probability of a rational
choice is at most $p+\epsilon$, where $p$ is the probability of a
rational choice for the majority GSWF on $n$ voters and three
alternatives.
\label{Cor:Maj-Stablest}
\end{corollary}

\begin{proof}
Similarly to the proof of Theorem~\ref{Cor:Neutral-Symmetric},
if $\alpha, \beta, \gamma \leq 1/4$ then
\[
W(f,g,h)=1/4 + \sum_{|S| \mbox{ odd }} \hat f(S)^2 (4
\alpha-1)^{|S|} + \sum_{|S| \mbox{ odd }} \hat f(S)^2 (4
\beta-1)^{|S|} + \sum_{|S| \mbox{ odd }} \hat f(S)^2 (4
\gamma-1)^{|S|} =
\]
\[
=1/4 - \sum_{|S| \mbox{ odd }} \hat f(S)^2 |4 \alpha-1|^{|S|} -
\sum_{|S| \mbox{ odd }} \hat f(S)^2 |4 \beta-1|^{|S|} - \sum_{|S|
\mbox{ odd }} \hat f(S)^2 |4 \gamma-1|^{|S|} =
\]
\[
= 1/4 - \langle\langle f,f \rangle\rangle_{|4 \alpha -1|} -
\langle\langle f,f \rangle\rangle_{|4 \beta -1|} - \langle\langle
f,f \rangle\rangle_{|4 \gamma -1|}.
\]
Hence, by Proposition~\ref{Prop:Maj-Stablest}, for every
$\epsilon>0$ there exists $n_0=n_0(\epsilon, \alpha, \beta,
\gamma)$ such that for every GSWF on $n>n_0$ voters satisfying the
assumptions of the corollary,
\[
W(f,g,h) \geq 1/4 - \frac{1}{2\pi} \arcsin (|4 \alpha -1|) -
\frac{1}{2\pi} \arcsin (|4 \beta -1|) - \frac{1}{2\pi} \arcsin (|4
\gamma -1|) - \epsilon.
\]
Finally, since for the majority GSWF $F_n$ on $n$ voters we have,
for all $0 \leq \rho \leq 1$,
\[
\lim_{n \rightarrow \infty} \langle\langle F_n,F_n
\rangle\rangle_{\rho} = \frac{1}{2\pi} \arcsin \rho
\]
(see~\cite{Mossel}, Section~4), the assertion of the corollary
follows.
\end{proof}

\begin{remark}
Corollary~\ref{Cor:Maj-Stablest} is proved in~\cite{Mossel} for a
uniform distribution of the preferences, as a corollary of the
``Majority is Stablest'' theorem.
\end{remark}

\begin{remark} Conjecture~8.1 of~\cite{Kalai-Choice} asserts that
for every distribution of the preferences (and even for more than
three alternatives), the probability of a rational choice for
GSWFs that are neutral, symmetric, and satisfy the IIA condition,
is maximized for the majority function. Hence,
Corollary~\ref{Cor:Maj-Stablest} proves {\it in the asymptotic
sense} (i.e., for a sufficiently large $n$) a special case of the
conjecture.
\end{remark}

We conclude this section by explaining the restriction on the
distribution of the individual preferences. The proof of
Theorem~\ref{Thm:Formula} crucially depends on the fact that $\hat
F_4^i(\{j\})$ vanishes for $j=1,2,3$. This condition holds if and
only if the probabilities of the preferences satisfy the following
three equations:
\[
\Pr[1,0,0]+\Pr[1,1,0]+\Pr[1,0,1]-\Pr[0,1,0]-\Pr[0,0,1]-\Pr[0,1,1]=0,
\]
\[
\Pr[0,1,0]+\Pr[1,1,0]+\Pr[0,1,1]-\Pr[1,0,0]-\Pr[0,0,1]-\Pr[1,0,1]=0,
\]
\[
\Pr[0,0,1]+\Pr[1,0,1]+\Pr[0,1,1]-\Pr[1,0,0]-\Pr[0,1,0]-\Pr[1,1,0]=0,
\]
where $\Pr[a,b,c]$ is a shorthand for $\Pr[(x_i,y_i,z_i)=(a,b,c)]$.
Summing the first two equations we get
\[
2\Pr[1,1,0]-2\Pr[0,0,1]=0,
\]
and similarly by summing the two other pairs of equations we get
$\Pr[1,0,1]=\Pr[0,1,0]$ and $\Pr[0,1,1]=\Pr[1,0,0]$. Finally, since
all the probabilities sum up to one, we get
$\Pr[1,0,0]+\Pr[0,1,0]+\Pr[0,0,1]=1/2$, and this completes the
restrictions described above. It is challenging to generalize
Theorem~\ref{Thm:Formula} to more general distributions on the
preferences, but the expression $\sum_{S \subset \{1,\ldots,3n\}}
\hat F_3(S) (\hat F_1(S) + \hat F_2(S))$ seems hard to compute in
the general case.

\section{Lower Bounds on the Probability of Rational Choice}
\label{Sec:Lower-Bounds}

In this section we establish lower bounds on the probability of
a rational choice for two classes of GSWFs: monotone balanced
functions and general balanced functions.

\subsection{Monotone Balanced GSWFs}

\begin{definition}
A function $f:\{0,1\}^n \rightarrow \mathbb{R}$ is monotone increasing if for
all $x=(x_1,\ldots,x_n)$ and $y=(y_1,\ldots,y_n)$,
\[
(\forall i: x_i \leq y_i) \Rightarrow (f(x) \leq f(y)).
\]
Similarly, a function is monotone decreasing if
\[
(\forall i: x_i \leq y_i) \Rightarrow (f(x) \geq f(y)).
\]
\end{definition}

\noindent Theorem~\ref{Thm:Monotone} is a special case of the
following, more general, result:
\begin{theorem}
Consider a GSWF on three alternatives satisfying the IIA condition
where the choice functions between the pairs of alternatives
$(A,B),(B,C),$ and $(C,A)$, denoted by $f,g,$ and $h$,
respectively, are monotone increasing. If the distribution of the preferences
is an even product distribution satisfying $\alpha,
\beta, \gamma \leq 1/4$ (and in particular, if the preferences are uniformly
distributed) then the probability of irrational choice satisfies:
\begin{equation}
W(f,g,h) \leq p_1 p_2 p_3 + (1-p_1)(1-p_2)(1-p_3),
\label{Eq4.0}
\end{equation}
where $p_1,p_2,$ and $p_3$ are the expectations of $f,g,$ and $h$,
respectively.
\label{Thm:Monotone'}
\end{theorem}

\begin{remark} The assertion of Theorem~\ref{Thm:Monotone'} is tight,
as can be seen in the following example: Assume that $f$ depends
only on the first voter, $g$ depends only on the second voter, and
$h$ depends only on the third voter. Then clearly, for all $-1
\leq \delta \leq 1$,
\[
\langle\langle f,g \rangle\rangle_{\delta} = \langle\langle g,h
\rangle\rangle_{\delta} = \langle\langle h,f
\rangle\rangle_{\delta} = 0,
\]
and thus,
\[
W(f,g,h) = p_1 p_2 p_3 + (1-p_1)(1-p_2)(1-p_3),
\]
where $p_1,p_2,$ and $p_3$ are the expectations of $f,g,$ and $h$,
respectively.
\end{remark}

\noindent By Theorem~\ref{Thm:Formula}, the assertion of
Theorem~\ref{Thm:Monotone'} is an immediate consequence of the following
proposition:
\begin{proposition}
For any two monotone increasing Boolean functions $f$ and $g$, and for every
$-1 \leq \delta \leq 1$,
\begin{equation}
\frac{1}{\delta} \langle\langle f,g \rangle\rangle_{\delta} \geq
0.
\label{Eq:Biased-Product1}
\end{equation}
\label{Prop:Biased-Product}
\end{proposition}

\noindent The proof of Proposition~\ref{Prop:Biased-Product} uses
properties of the Bonamie-Beckner noise operator and the FKG correlation
inequality~\cite{FKG}. For the reader's convenience, we recall the statement
of the FKG inequality in the special case
of the uniform measure on the discrete cube.
\begin{theorem}[Fortuin, Kasteleyn, and Ginibre]
Consider the discrete cube $\{0,1\}^n$ endowed with the uniform measure
$\mu$, and let $f,g:\{0,1\}^n \rightarrow \mathbb{R}$. Then:
\begin{enumerate}
\item If both $f$ and $g$ are monotone increasing, then
$\mathbb{E}_{\mu}[fg] \geq \mathbb{E}_{\mu}[f] \mathbb{E}_{\mu}[g]$.

\item If $f$ is monotone increasing and $g$ is monotone
decreasing, then $\mathbb{E}_{\mu}[fg] \leq \mathbb{E}_{\mu}[f]
\mathbb{E}_{\mu}[g]$.
\end{enumerate}
\label{Thm:FKG}
\end{theorem}

\noindent \textbf{Proof of Proposition~\ref{Prop:Biased-Product}.} By
the definition of the noise operator $T_{\epsilon}$, we have
\[
\frac{1}{\delta} \langle\langle f,g \rangle\rangle_{\delta} =
\frac{1}{\delta} \sum_{S \neq \emptyset} \delta^{|S|} \hat f(S) \hat g(S)
= \frac{1}{\delta}\sum_{S \neq \emptyset} \widehat{T_{\delta} f} (S)
\hat g(S) .
\]
By the Parseval identity,
\[
\frac{1}{\delta}\sum_{S \neq \emptyset} \widehat{T_{\delta} f} (S)
\hat g(S) = \frac{1}{\delta} \Big(\sum_{S} \widehat{T_{\delta} f} (S)
\hat g(S) - \widehat{T_{\delta} f} (\emptyset) \hat g(\emptyset) \Big) =
\frac{1}{\delta}(\mathbb{E}_{\mu}[T_{\delta}f \cdot
g] - \mathbb{E}_{\mu}[T_{\delta}f] \mathbb{E}_{\mu}[g]).
\]
Hence, Inequality~(\ref{Eq:Biased-Product1}) is equivalent to the inequality:
\begin{equation}
\frac{1}{\delta}(\mathbb{E}_{\mu}[T_{\delta}f \cdot g] -
\mathbb{E}_{\mu}[T_{\delta}f] \mathbb{E}_{\mu}[g]) \geq 0.
\label{Eq:Biased-Product2}
\end{equation}
Since the function $g$ is monotone increasing, Inequality~(\ref{Eq:Biased-Product2})
will follow from the FKG inequality, once we show that $T_{\delta} f$ is monotone
increasing if $0 \leq \delta \leq 1$, and monotone decreasing if $-1 \leq \delta \leq 0$.
We show the case of $-1 \leq \delta \leq 0$ (the case of positive $\delta$ is similar).

\medskip

\noindent Without loss of generality, it is sufficient to prove that for
all $(x_2,\ldots,x_n) \in \{0,1\}^{n-1}$,
\begin{equation}
T_{\delta} f (0,x_2,\ldots,x_n) \geq T_{\delta} f
(1,x_2,\ldots,x_n).
\label{Eq4.2.1}
\end{equation}
Using the equivalent definition of the noise operator presented
in Section~\ref{sec:sub:noise} (i.e., Equation~(\ref{Eq:Beckner1})),
\[
T_{\delta} f(x_1,x_2,\ldots,x_n) = \mathbb{E} [f((x_1,x_2,\ldots,x_n)
\oplus (y_1,\ldots,y_n))],
\]
where each $y_i$ is distributed according to the distribution
$\Pr[y_i=0]=(1+\delta)/2, \Pr[y_i=1]=(1-\delta)/2$,
independently of other $y_i$'s. Thus, we have to show that for all
$(x_2,\ldots,x_n) \in \{0,1\}^{n-1}$,
\[
\mathbb{E} [f((y_1,x_2 \oplus y_2,\ldots,x_n \oplus y_n))] \geq
\mathbb{E} [f((1 \oplus y_1,x_2 \oplus y_2,\ldots,x_n \oplus y_n))].
\]
Therefore, it is sufficient to show that for each
$(z_2,\ldots,z_n) \in \{0,1\}^{n-1}$,
\[
\mathbb{E}_{y_1} [f((y_1,z_2,\ldots,z_n))] \geq
\mathbb{E}_{y_1} [f((1 \oplus y_1,z_2,\ldots,z_n))],
\]
or equivalently
\[
\frac{1+\delta}{2} f(0,z_2,\ldots,z_n) + \frac{1-\delta}{2} f(1,z_2,\ldots,z_n) \geq
\frac{1-\delta}{2} f(0,z_2,\ldots,z_n) + \frac{1+\delta}{2} f(1,z_2,\ldots,z_n).
\]
This inequality indeed follows from the monotonicity of $f$, since $\delta \leq 0$.
This completes the proof of Proposition~\ref{Prop:Biased-Product}.
$\square$

\medskip

\noindent For a general even product distribution of the preferences, the
probability of a rational choice for balanced monotone choice functions
can be as low as $1/2$ (compared to $3/4$ in the case
$\alpha,\beta,\gamma \leq 1/4$). An example in which the probability is
$1/2$ is the following:

\medskip

\begin{example}
Assume that the distribution on the preferences is: $\Pr[(x_i,y_i,z_i)=(1,1,0)]=1/2$
and $\Pr[(x_i,y_i,z_i)=(0,0,1)]=1/2$, while the probability of the
other preferences is zero (i.e., $\alpha=1/2$ and
$\beta=\gamma=0$). The choice functions $f$ and $g$ are a
dictatorship of the first voter, and $h$ is a dictatorship of the
second voter. Then it is easy to see that $W(f,g,h)=1/2$.
\end{example}

\noindent It can be shown that $1/2$ is a lower bound for the
probability of a rational choice in our case. Indeed, by
Theorem~\ref{Thm:Formula}, for balanced choice functions we have
\begin{equation}
W(f,g,h)= 1/4 + \langle\langle f,g \rangle\rangle_{4 \alpha -1} +
\langle\langle g,h \rangle\rangle_{4 \beta -1}
+ \langle\langle h,f \rangle\rangle_{4 \gamma -1}.
\label{Eq4.1.4.1}
\end{equation}
By Proposition~\ref{Prop:Biased-Product}, an expression of the
form $\langle\langle f,g \rangle\rangle_{4 \alpha -1}$ can be
positive only if $\alpha > 1/4$. Since in
our distribution $\alpha+\beta+\gamma=1/2$, at most one of the
expressions of this form appearing in Equation~(\ref{Eq4.1.4.1}) is
positive. By the Cauchy-Schwarz inequality,
\[
\langle\langle f,g \rangle\rangle_{4 \alpha -1} = \sum_{S \neq \emptyset} \hat f(S) \hat g(S) (4 \alpha-1)^{|S|} \leq 1/4,
\]
and similarly for $\beta$ and $\gamma$. Therefore,
$W(f,g,h) \leq 1/4 + (1/4 +0 +0) =1/2$.

\medskip

\noindent The probability of a rational choice is equal to $1/2$ if
and only if $\langle\langle f,g \rangle\rangle_{4 \alpha -1}=1/4$, and
$\langle\langle g,h \rangle\rangle_{4 \beta -1} =
\langle\langle h,f \rangle\rangle_{4 \gamma -1} = 0$ (up to
a permutation between $\alpha,\beta,$ and $\gamma$). By the
Cauchy-Schwarz inequality, this occurs if and only if the following
three conditions are satisfied:
\begin{itemize}
\item The distribution of the preferences is $\alpha=1/2,
\beta=\gamma=0$.
\item The choice functions $f,g$ satisfy $f=g$.
\item The choice function $h$ is independent of $f$, in the
following sense: The set of voters $\{1,\ldots,n\}$ can be partitioned
into two disjoint sets $A$ and $B$ such that the output of $f$
depends only on the elements of $A$, and the output of $h$ depends
only on the elements of $B$.
\end{itemize}

\subsection{General Balanced GSWFs}

In~\cite{Kalai-Choice} it is stated (Proposition 5.2) that if the
preferences are uniformly distributed, then the lower bound for
the probability of rational choice for general balanced GSWFs is
$2/3$. However, the proof sketched in~\cite{Kalai-Choice} is
insufficient\footnote{The proof in~\cite{Kalai-Choice} assumes
implicitly that the least possible probability is achieved when
the Fourier-Walsh coefficients of the functions $f,g,h$ are concentrated
on the second level. It is not clear whether this assumption
is correct.}, and it is not even clear that the lower bound itself is
correct. In this subsection we prove a weaker lower bound, and
discuss its tightness.

\begin{theorem}
Consider a GSWF on three alternatives satisfying the IIA condition
such that the choice functions between the pairs of alternatives
are balanced. If the preferences are uniformly distributed then the
probability of a rational choice is at least $5/8$.
\label{Thm:Balanced}
\end{theorem}

\begin{proof}
Consider the Fourier-Walsh expansions of the choice functions $f,g$, and $h$. Let
\[
\sum_{i=1}^n \hat f(\{i\})^2=a, \qquad \sum_{i=1}^n \hat g(\{i\})^2=b, \qquad \sum_{i=1}^n \hat h(\{i\})^2=c.
\]
Since $f,g,$ and $h$ are balanced, then by the Parseval identity
\[
\sum_{|S|>1} \hat f(S)^2 = 1/4-a, \qquad \sum_{|S|>1} \hat g(S)^2 = 1/4-b, \qquad \sum_{|S|>1} \hat h(S)^2 =
1/4-c.
\]
Recall that by Theorem~\ref{Thm:Formula}, in our case
\begin{equation}
W(f,g,h)= 1/4 + \langle\langle f,g \rangle\rangle_{-1/3} + \langle\langle g,h \rangle\rangle_{-1/3} +
\langle\langle h,f \rangle\rangle_{-1/3}.
\label{Eq4.2.2.1}
\end{equation}
We have
\[
\langle \langle f,g \rangle \rangle_{-1/3} + \langle \langle g,h \rangle \rangle_{-1/3} + \langle \langle h,f
\rangle \rangle_{-1/3} = \sum_{|S|>0} \Big(\hat f(S) \hat g(S) + \hat g(S) \hat h(S) + \hat h(S) \hat f(S) \Big)
(-1/3)^{|S|}=
\]
\[
=-\frac{1}{3}\sum_{|S|=1} \Big(\hat f(S) \hat g(S) + \hat g(S) \hat h(S) + \hat h(S) \hat f(S)\Big) + \sum_{|S|>1} \Big(\hat f(S) \hat g(S) + \hat g(S) \hat h(S) + \hat h(S) \hat f(S)\Big) (-1/3)^{|S|} \leq
\]
\[
\leq -\frac{1}{3} \sum_{|S|=1} \Big(\hat f(S) \hat g(S) + \hat g(S) \hat h(S) + \hat h(S) \hat f(S) \Big) +
\frac{1}{9} \sum_{|S|>1} \Big|\hat f(S) \hat g(S) + \hat g(S) \hat h(S) + \hat h(S) \hat f(S) \Big|.
\]
In order to bound the first summand, we use the elementary
inequality
\[
-(xy+yz+xz) \leq (x^2+y^2+z^2)/2.
\]
We get
\[
-\frac{1}{3}\sum_{|S|=1} \Big(\hat f(S) \hat g(S) + \hat g(S) \hat h(S) + \hat h(S) \hat f(S) \Big) =
-\frac{1}{3} \sum_{i=1}^n \Big(\hat f(\{i\}) \hat g(\{i\}) + \hat g(\{i\}) \hat h(\{i\}) +
\hat h(\{i\}) \hat f(\{i\})\Big) \leq
\]
\[
\leq \frac{1}{6} \sum_{i=1}^n \Big(\hat f(\{i\})^2 + \hat g(\{i\})^2 + \hat h(\{i\})^2 \Big) = \frac{a+b+c}{6}.
\]
In order to bound the second summand, we use the Cauchy-Schwarz
inequality and the inequality between the arithmetic and the
geometric means. Let
\[
\tilde{f}=\sum_{|S|>1} |\hat f(S)| r_S, \qquad
\tilde{g}=\sum_{|S|>1} |\hat g(S)| r_S.
\]
Applying the Cauchy-Schwarz inequality and the Parseval identity we get
\[
\sum_{|S|>1} |\hat f(S) \hat g(S)|  = \langle \tilde{f},\tilde{g}
\rangle \leq ||\tilde{f}||_2 ||\tilde{g}||_2 =
\sqrt{(1/4-a)(1/4-b)} \leq 1/4 - (a+b)/2,
\]
where the last inequality follows from the inequality between the
arithmetic and the geometric means. Applying the same inequalities
to the pairs $(g,h)$ and $(h,f)$, we get
\[
\frac{1}{9} \sum_{|S|>1} \Big|\hat f(S) \hat g(S) + \hat g(S) \hat h(S) + \hat h(S) \hat f(S) \Big| \leq
\]
\[
\leq \frac{1}{9} \Big(\frac{1}{4} - \frac{a+b}{2} + \frac{1}{4} - \frac{b+c}{2} +
\frac{1}{4} - \frac{c+a}{2} \Big) = \frac{1}{12} - \frac{a+b+c}{9}.
\]
Combining the bounds obtained above, we get
\[
-\frac{1}{3}\sum_{|S|=1} \Big(\hat f(S) \hat g(S) + \hat g(S) \hat h(S) + \hat h(S) \hat f(S) \Big) +
\frac{1}{9} \sum_{|S|>1} \Big|\hat f(S) \hat g(S) + \hat g(S) \hat h(S) + \hat h(S) \hat f(S) \Big| \leq
\]
\[
\leq \frac{a+b+c}{6} + \frac{1}{12} - \frac{a+b+c}{9} = \frac{1}{12} + \frac{a+b+c}{18}.
\]
Substitution to Equation~(\ref{Eq4.2.2.1}) yields:
\[
W(f,g,h) \leq 1/4 + 1/12 + (a+b+c)/18 = 1/3+(a+b+c)/18.
\]
Finally, since by the Parseval identity we have $0 \leq a,b,c \leq 1/4$, the maximum
in the right hand side is obtained for $a=b=c=1/4$, and thus,
\[
W(f,g,h) \leq 1/3+(3/4)/18 = 3/8,
\]
as asserted.
\end{proof}

\noindent The tightness of the lower bound in
Theorem~\ref{Thm:Balanced} is not clear to us. The example
presented in~\cite{Kalai-Choice} yields the value $W(f,g,h)=1/3$,
where all the Fourier-Walsh coefficients of $f,g,$
and $h$ are concentrated on the second level. Another example
yielding the same value of $W(f,g,h)$ is
\[
f(x_1,\ldots,x_n)=x_i, \qquad g(x_1,\ldots,x_n)=x_i, \qquad h(x_1,\ldots,x_n)=1-x_i,
\]
for any $1\leq i \leq n$. In this example, all the weight of
$f,g,$ and $h$ is concentrated on the first level. It seems
possible that the correct lower bound is $2/3$, as asserted
in~\cite{Kalai-Choice}. However, in order to prove this bound, one
has to exploit the fact that the choice functions are {\it
Boolean}, as can be seen in the following example:

\medskip

\begin{example}
Let $f,g,h$ be defined by $\hat f(\emptyset) = \hat g(\emptyset) =
\hat h(\emptyset)=1/2$ and
\[
\hat f({i}) = \frac{2}{2\sqrt{6}}, \qquad \hat f({j}) =
-\frac{1}{2\sqrt{6}}, \qquad \hat f({k}) = -\frac{1}{2\sqrt{6}},
\]
\[
\hat g({i}) = -\frac{1}{2\sqrt{6}}, \qquad \hat g({j}) =
\frac{2}{2\sqrt{6}}, \qquad \hat g({k}) = -\frac{1}{2\sqrt{6}},
\]
\[
\hat h({i}) = -\frac{1}{2\sqrt{6}}, \qquad \hat h({j}) =
-\frac{1}{2\sqrt{6}}, \qquad \hat h({k}) = \frac{2}{2\sqrt{6}},
\]
for $1 \leq i<j<k \leq n$. The rest of the Fourier-Walsh
coefficients of $f,g,$ and $h$ are zero. Since
\[
\sum_{S \neq \emptyset} \hat f(S)^2 = \sum_{S \neq \emptyset} \hat
g(S)^2 = \sum_{S \neq \emptyset} \hat h(S)^2 = 1/4,
\]
the functions $f,g,$ and $h$ ``look like'' balanced functions from
the Fourier-theoretic point of view. Nevertheless, $W(f,g,h)=3/8$,
which agrees with the lower bound of Theorem~\ref{Thm:Balanced}.
This shows that in order to improve Theorem~\ref{Thm:Balanced}, we
have to use the fact that $f,g,$ and $h$ are Boolean functions.
\end{example}

\section{Upper Bounds on the Probability of Rational Choice}
\label{Sec:Upper-Bounds}

Throughout this section we assume that the preferences are
uniformly distributed.

\medskip

\noindent In this section we discuss Kalai's~\cite{Kalai-Choice}
proof of Arrow's Impossibility theorem for neutral GSWFs on three
alternatives. First we discuss the possibility of extending
Kalai's proof to other special cases of Arrow's theorem, and then
we discuss the stability version of the theorem proved by Kalai
(for neutral GSWFs).

\subsection{Extending Kalai's Proof to Other Special Cases of Arrow's
Theorem}

Kalai's proof uses the Fourier-theoretic formula for the
probability of irrational choice for GSWFs on three altrenatives
satisfying the IIA condition (Theorem~\ref{Thm:Formula}). For a
balanced GSWF, the formula reads:
\begin{equation}\label{Eq-New:1}
W(f,g,h)= 1/4 + \langle\langle f,g \rangle\rangle_{-1/3} +
\langle\langle g,h \rangle\rangle_{-1/3} + \langle\langle h,f
\rangle\rangle_{-1/3}.
\end{equation}
Define
\[
\tilde{f}=\sum_{S \neq \emptyset} \hat f(S) r_S, \qquad
\overline{g}=\sum_{S \neq \emptyset} \hat g(S) (-1/3)^{|S|} r_S.
\]
Note that since $f$ and $g$ are balanced, by the Parseval identity
$||\tilde{f}||_2=1/2$ and $||\overline{g}||_2 \leq 1/6$.
Therefore, by the Cauchy-Schwarz inequality,
\[
|\langle\langle f,g \rangle\rangle_{-1/3}| = |\langle
\tilde{f},\overline{g} \rangle| \leq ||\tilde{f}||_2
||\overline{g}||_2 \leq 1/12,
\]
and it can be shown that equality can hold only if all the
Fourier-Walsh coefficients of $f$ and of $g$ are on the first
level. Then, it can be further shown that $W(f,g,h)=0$ can hold
only if $f,g,$ and $h$ are dictatorships of the same voter, and
this completes the proof of the theorem.

It was suggested in~\cite{Kalai-Choice} to use the same reasoning
in the non-balanced case. Such generalization is possible if
$p_1,p_2,$ and $p_3$, the expectations of $f,g,$ and $h$, satisfy
some condition described in~\cite{Kalai-Choice}. However, this
condition is not satisfied in many cases, e.g., for $p_1=p_2=1/5$
and $p_3=1$, as noted in~\cite{Kalai-Choice}.
Kalai~\cite{Kalai-Private} suggested to improve the upper bound
$||\overline{g}||_2 \leq 1/6$ (or, more generally,
$||\overline{g}||_2 \leq \sqrt{p_2(1-p_2)}/3$) used in the proof
by using the Bonamie-Beckner hypercontractive
inequality~\cite{Bonamie,Beckner}.

We show by an example that this proof strategy, even using the
hypercontractive inequality, cannot lead to a complete proof of
Arrow's theorem. The example shows that if the biased inner
product $\langle\langle f,g \rangle\rangle_{-1/3}$ is replaced by
\[
\langle\langle f,g \rangle\rangle'_{-1/3}=-\sum_{S \neq \emptyset}
|\hat f(S) \hat g(S) (-1/3)^{|S|}|,
\]
then there exist functions $f,g,h$ such that
\[
W'(f,g,h)=p_1 p_2 p_3 + (1-p_1)(1-p_2)(1-p_3) + \langle\langle f,g
\rangle\rangle'_{-1/3} + \langle\langle g,h \rangle\rangle'_{-1/3}
+ \langle\langle h,f \rangle\rangle'_{-1/3} < 0.
\]
Hence, a proof of Arrow's theorem using Equation~(\ref{Eq-New:1})
cannot ignore the {\it sign} of the Fourier-Walsh coefficients of
the choice functions.

\medskip

\noindent The example uses the notion of a {\it dual} function:
\begin{definition}
Let $f:\{0,1\}^n \rightarrow \{0,1\}$. The dual function of $f$ (which
we denote by $f':\{0,1\}^n \rightarrow \{0,1\}$), is defined by
\[
f'(x_1,x_2,\ldots,x_n)=1-f(1-x_1,1-x_2,\ldots,1-x_n).
\]
\end{definition}
\noindent The Fourier-Walsh expansion of the dual function is
closely related to the expansion of the original function:
\begin{claim}
Consider the Fourier-Walsh expansions of a Boolean function $f$
and its dual function $f'$. For all $S \subset \{1,\ldots,n\}$
with $|S| \geq 1$,
\[
\hat f'(S) = (-1)^{|S|-1} \hat f(S).
\]
\label{Claim5.1}
\end{claim}
\noindent The simple proof of the claim is omitted.

\medskip

\begin{example}
Assume that $n$ is odd, $f(x_1,x_2,\ldots,x_n)=x_1 \cdot x_2 \cdot
\ldots \cdot x_n$ is the AND function, $g=f'$ is its dual
function, and $h$ is the majority function. We have
\[
p_1=\mathbb{E}[f]=2^{-n}, \qquad p_2=\mathbb{E}[g]=1-2^{-n},
\qquad p_3=\mathbb{E}[h]=1/2.
\]
The Fourier-Walsh coefficients of $f$ satisfy $|\hat f(S)|=2^{-n}$
for all $S \subset \{1,\ldots,n\}$. The first-level Fourier-Walsh
coefficients of the majority function are
\[
\hat h(\{i\}) = {{n-1}\choose{(n-1)/2}}2^{-n} \approx
\sqrt{\frac{1}{2 \pi n}}
\]
for all $1 \leq i \leq n$. Hence,
\[
\langle\langle h,f \rangle\rangle'_{-1/3} \leq -\frac{1}{3} \sum_{i=1}^n
|\hat h(\{i\}) \hat f(\{i\})| \approx -\frac{1}{3} n 2^{-n}
\sqrt{\frac{1}{2 \pi n}} = -\frac{1}{3 \sqrt{2 \pi}} \sqrt{n}
2^{-n}.
\]
Therefore,
\[
W'(f,g,h) \leq p_1 p_2 p_3 + (1-p_1)(1-p_2)(1-p_3) +
\langle\langle h,f \rangle\rangle'_{-1/3} \leq 2^{-n}(1-2^{-n})
-\frac{1}{3 \sqrt{2}} \sqrt{n} 2^{-n} < 0,
\]
for $n$ large enough.
\end{example}

\bigskip

A possible step towards a Fourier-theoretic proof of Arrow's
theorem in the general case is the following lower bound on the
biased inner product $\langle\langle f,g \rangle\rangle_{\delta}$:
\begin{proposition}
Let $f,g:\{0,1\}^n \rightarrow \mathbb{R}_{+}$ be non-negative
functions with $\mathbb{E}[f]=p_1$ and $\mathbb{E}[g]=p_2$, and
let $-1 \leq \delta \leq 1$. Then
\[
\langle\langle f,g \rangle\rangle_{\delta} \geq -p_1 p_2,
\]
and equality holds if and only if either $f \equiv 0$ or $g \equiv 0$.
\label{Prop:Lower-Bound-Biased-Product}
\end{proposition}

\begin{proof}
We prove the proposition in the case $\delta<0$, the case $\delta
\geq 0$ is similar. Let
$f''(x_1,\ldots,x_n)=f(1-x_1,\ldots,1-x_n)$. Clearly, $\hat
f''(\emptyset) = \mathbb{E}[f''] = p_1$. By Claim~\ref{Claim5.1},
for all $S \neq \emptyset$,
\[
\hat f''(S)= (-1)^{|S|} \hat f(S).
\]
Hence, by the definition of the Bonamie-Beckner noise operator,
\[
\widehat{T_{-\delta} f''} (S) = (-\delta)^{|S|} (-1)^{|S|} \hat
f(S) = \delta^{|S|} \hat f(S).
\]
Therefore, by the Parseval identity,
\[
\langle\langle f,g \rangle\rangle_{\delta} + p_1 p_2 = \sum_{S
\neq \emptyset} \hat f(S) \hat g(S) \delta^{|S|} + p_1 p_2 =
\sum_{S \neq \emptyset} \widehat{T_{-\delta} f''} (S) \hat g(S) +
\widehat{T_{-\delta} f''} (\emptyset) \hat g(\emptyset) = \langle
T_{-\delta} f'' , g \rangle.
\]
Finally, by the assumption $g$ is non-negative, and by
Equation~(\ref{Eq:Beckner1}), the function $T_{-\delta} f''$ is
strictly positive, unless $f \equiv 0$. Hence,
\[
\langle T_{-\delta} f'' , g \rangle = \frac{1}{2^n}\sum_{x \in
\{0,1\}^n} T_{-\delta} f''(x) g(x) > 0,
\]
unless either $f \equiv 0$ or $g \equiv 0$, and in that cases
$\langle T_{-\delta} f'' , g \rangle = 0$. This completes the
proof of the proposition.
\end{proof}

\begin{corollary}
The assertion of Arrow's theorem holds if $p_1+p_2+p_3 \leq 1$,
where $p_1,p_2,$ and $p_3$ are the expectations of the choice
functions $f,g,$ and $h$.
\end{corollary}

\begin{proof}
By Proposition~\ref{Prop:Lower-Bound-Biased-Product},
\[
\langle\langle f,g \rangle\rangle_{-1/3} + \langle\langle g,h
\rangle\rangle_{-1/3} + \langle\langle h,f
\rangle\rangle_{-1/3} > -(p_1 p_2 + p_2 p_3 + p_3 p_1).
\]
(Equality cannot hold since by the assumption of Arrow's theorem,
$f,g,$ and $h$ are non-constant). Hence, by
Equation~(\ref{Eq3.0}),
\[
W(f,g,h) > p_1 p_2 p_3 + (1-p_1)(1-p_2)(1-p_3) - (p_1 p_2 + p_2
p_3 + p_3 p_1) = 1 - p_1 - p_2 - p_3 \geq 0,
\]
and thus the assertion of Arrow's theorem holds.
\end{proof}

\noindent Another corollary of
Proposition~\ref{Prop:Lower-Bound-Biased-Product} uses dual
functions:
\begin{corollary}
Let $f,g:\{0,1\}^n \rightarrow \{0,1\}$ such that
$\mathbb{E}[f]=p_1$ and $\mathbb{E}[g]=p_2$, and let $-1 \leq
\delta \leq 1$. Then
\[
\langle\langle f,g \rangle\rangle_{\delta} \geq -(1-p_1)(1-p_2),
\]
and equality holds if and only if either $f \equiv 1$ or $g \equiv 1$.
\label{Cor:Lower-Bound-Biased-Product}
\end{corollary}

\begin{proof}
Denote the dual functions of $f$ and $g$ by $f'$ and $g'$,
respectively. By Claim~\ref{Claim5.1}, for all $S \neq \emptyset$,
\[
\hat f'(S) \hat g'(S) = (-1)^{|S|-1} \hat f(S) (-1)^{|S|-1} \hat
g(S) = \hat f(S) \hat g(S),
\]
and hence
\[
\langle\langle f',g' \rangle\rangle_{\delta} = \langle\langle f,g
\rangle\rangle_{\delta}.
\]
The functions $f',g'$ are non-negative and satisfy
$\mathbb{E}[f']=1-p_1$ and $\mathbb{E}[g']=1-p_2$. Thus, by
Proposition~\ref{Prop:Lower-Bound-Biased-Product},
\[
\langle\langle f',g' \rangle\rangle_{\delta} \geq -(1-p_1)(1-p_2),
\]
and equality holds if and only if $f' \equiv 0$ or $g' \equiv 0$, or
equivalently, if and only if $f \equiv 1$ or $g \equiv 1$.
\end{proof}

Proposition~\ref{Prop:Lower-Bound-Biased-Product} and
Corollary~\ref{Cor:Lower-Bound-Biased-Product} yield an immediate
proof of Arrow's theorem in the case where there exists $1 \leq i
\leq 3$ such that $p_i=0$ or $p_i=1$. Indeed, two of the biased
inner products of the form $\langle\langle f,g
\rangle\rangle_{\delta}$ appearing in Equation~(\ref{Eq3.0})
vanish, and the third biased inner product can be bounded using
either Proposition~\ref{Prop:Lower-Bound-Biased-Product} or
Corollary~\ref{Cor:Lower-Bound-Biased-Product}. This settles the
example given in~\cite{Kalai-Choice}. However, we note that this
case is anyway ruled out by the assumption (made in Arrow's
theorem) that the choice functions are non-constant.

\medskip

It seems possible that Kalai's proof and
Proposition~\ref{Prop:Lower-Bound-Biased-Product} can be extended
to a proof of broader special cases of Arrow's theorem. Such
extension is of interest even after the recent analytic proof of
Arrow's theorem (in the general case) by
Mossel~\cite{Mossel-Arrow}, since in the cases where Kalai's proof
applies, the same argument yields a stability version of the
theorem in which the dependence of $\delta(\epsilon)$ on
$\epsilon$ is linear, while the dependence in Mossel's theorem is
much weaker.

\subsection{Discussion on a Stability Version of Arrow's Theorem}

In~\cite{Kalai-Choice}, Kalai proved a stability version of
Arrow's theorem:
\begin{theorem}[\cite{Kalai-Choice}]
For every $\epsilon>0$ and for every balanced GSWF on three
alternatives, if the probability that the social choice is
irrational is smaller than $\epsilon$ then there is a dictator
such that the probability that the output of the GSWF differs from
the dictator's choice is smaller than $K \cdot \epsilon$, where
$K$ is a universal constant.
\label{Thm:Gil-Stability}
\end{theorem}
\noindent Following Theorem~\ref{Thm:Gil-Stability}, it is natural
to ask:

\begin{question}
Amongst the GSWFs on three alternatives satisfying the assumptions
of Arrow's theorem, which is the ``most rational'' one (i.e., the
one having the highest probability of a rational outcome)?
\label{Q1}
\end{question}

\begin{remark}
The idea behind the question is similar to the idea behind the
Hilton-Milner theorem~\cite{Hilton-Milner} concerning intersecting
families. A family of subsets of a given finite set is called {\it
intersecting} if the intersection of any two elements of the
family is non-empty. The Erd\"{o}s-Ko-Rado
theorem~\cite{Erdos-Ko-Rado} asserts that an intersecting family
of $k$-element subsets of an $n$-element set has at most
${{n-1}\choose{k-1}}$ elements, and that the only maximal families
are of the form $\{S \subset \{1,\ldots,n\}: |S|=k, i \in S\}$,
for $1 \leq i \leq n$. The Hilton-Milner
theorem~\cite{Hilton-Milner} answers the question: What is the
{\it second largest} intersecting family?

Similarly, in our situation, Arrow's theorem asserts that under
some conditions, the only ``most rational'' GSWFs are the
dictatorship functions. Question~\ref{Q1} asks, what is the most
rational GSWF except for the dictatorship functions.
\end{remark}

One class of natural candidates for being the most rational GSWF
is functions close to a dictatorship. Since the probability that
the output of the GSWF differs from a dictatorship is at least
$2^{-n}$, Theorem~\ref{Thm:Gil-Stability} implies that for every balanced
function of this class, the probability of irrational choice is at
least $K^{-1} \cdot 2^{-n}$, where $K$ is a universal constant.

Another class of natural candidates is almost constant
functions. It can be shown that if all the three choice functions
are almost constant (e.g., $f(x_1,\ldots,x_n)=1$ unless
$(x_1,\ldots,x_n)=(0,0,\ldots,0)$) then the probability of
irrational choice is also $\Theta(2^{-n})$.

However, it appears that there exists a GSWF with a much lower
probability of irrational outcome:

\begin{example}
Assume that $n$ is odd, $f(x_1,x_2,\ldots,x_n)=x_1 \cdot x_2 \cdot
\ldots \cdot x_n$ is the AND function, $g=f'$ is its dual
function, and $h$ is the majority function. Let
\[
p_1=\mathbb{E}[f]=2^{-n}, \qquad p_2=\mathbb{E}[g]=1-2^{-n},
\qquad p_3=\mathbb{E}[h]=1/2.
\]
By the proof of Proposition~\ref{Prop:Lower-Bound-Biased-Product},
\[
\langle\langle f,g \rangle\rangle_{-1/3} = \langle T_{1/3} f'' , g
\rangle - p_1 p_2.
\]
By Equation~(\ref{Eq:Beckner1}),
\[
\langle T_{1/3} f'' , g \rangle = 2^{-n} \sum_{x \in \{0,1\}^n}
\Big( \frac{1}{3} \Big)^{\sum_{i=1}^n x_i} \Big(\frac{2}{3} \Big)^{n - \sum_{i=1}^n x_i} g(x) =
2^{-n} (1-\Big(\frac{2}{3}\Big)^n),
\]
and thus,
\[
\langle\langle f,g \rangle\rangle_{-1/3} = 2^{-n} (1-(2/3)^n) -
2^{-n} (1-2^{-n}) = -(1/3)^n +(1/4)^n.
\]
Similarly,
\[
\langle T_{1/3} f'' , h \rangle = 2^{-n} \sum_{x \in \{0,1\}^n}
\Big( \frac{1}{3} \Big)^{\sum_{i=1}^n x_i}
\Big(\frac{2}{3} \Big)^{n - \sum_{i=1}^n x_i} h(x) =
\]
\[
= 2^{-n} \sum_{\{x:\sum_{i=1}^n x_i > n/2\}}
\Big( \frac{1}{3} \Big)^{\sum_{i=1}^n x_i}
\Big(\frac{2}{3} \Big)^{n - \sum_{i=1}^n x_i} \leq 2^{-n}
\sum_{\{x:\sum_{i=1}^n x_i > n/2\}} \Big(\frac{1}{3}\Big)^{n/2}
\Big(\frac{2}{3}\Big)^{n/2} =
\]
\[
= \frac{1}{2} \Big(\frac{2}{9}\Big)^{n/2} \approx \frac{1}{2} \cdot 0.471^n.
\]
Hence,
\[
\langle\langle f,h \rangle\rangle_{-1/3} \leq \frac{1}{2} \cdot 0.471^n -
\frac{1}{2} \cdot 2^{-n}.
\]
Finally, since the dual function of $f$ is $g$ and since $h$ is
self-dual,
\[
\langle\langle g,h \rangle\rangle_{-1/3} = \langle\langle f,h
\rangle\rangle_{-1/3}.
\]
Therefore,
\[
W(f,g,h)=p_1 p_2 p_3 + (1-p_1)(1-p_2)(1-p_3) + \langle\langle f,g
\rangle\rangle_{-1/3} + \langle\langle g,h \rangle\rangle_{-1/3} +
\langle\langle h,f \rangle\rangle_{-1/3} \leq
\]
\[
\leq 2^{-n}(1-2^{-n}) + 0.471^n - 2^{-n} - (1/3)^n + (1/4)^n \leq
0.471^n.
\]
\end{example}

We conjecture that the GSWF in the example is the most rational
GSWF under the conditions of Arrow's theorem, but we weren't able
to prove this conjecture.

The example can be generalized to a series of examples that proves
Theorem~\ref{Thm:Instability}.

\medskip

\begin{example}
For $0<q<1/2$, for any $K>0$, and for an odd $n$, let
\[
f(x) = \left\lbrace
\begin{array}{c l}
    1, & \sum_{i=1}^n x_i \geq (1-q)n \\
    0, & \sum_{i=1}^n x_i < (1-q)n,
  \end{array}
\right.
\]
$g$ is the dual function of $f$, and $h$ is the majority function.
We use the well-known (see, for example,~\cite{Upfal}, Lemma 9.2)
inequality:
\begin{equation}
\frac{2^{n H(q)}}{n+1} \leq {{n}\choose{qn}} \leq 2^{n H(q)},
\label{Eq5.2.1}
\end{equation}
where $H(q)=-q \log_2 q - (1-q) \log_2 (1-q)$ is the value of the
entropy function at $q$. By Inequality~(\ref{Eq5.2.1}), we have
\[
\min(\mathbb{E}[f], \mathbb{E}[g], \mathbb{E}[h]) \geq \frac{2^{n
H(q)-1}}{n+1}.
\]
Hence, amongst any pair of alternatives, the probability of each alternative
to be preferred by the society over the other alternative is at least
$\eta=2^{n H(q)-1}/(n+1)$. On the other hand, using considerations
similar to those of the previous example (but more tedious), one obtains
\[
W(f,g,h) \leq 2^{n (H(q)-1)}(1-2^{n (H(q)-1)}) + 2^{n (q-1.08)} -
2^{n (H(q)-1)} < 2^{n (q-1.08)}.
\]
Since for all $q<1/2$,
\[
q-1.08 < H(q)-1,
\]
for $n=n(q,K)$ big enough we have
\[
W(f,g,h)<2^{n (q-1.08)} < \frac{2^{n (H(q)-1)}}{(n+1)K} = \frac{\eta}{K}.
\]
Therefore, substituting $\epsilon = 1-H(q)$, the assertion of
Theorem~\ref{Thm:Instability} follows.\footnote{We note that a
stronger bound on $W(f,g,h)$ for this example can be deduced from
the computation in (\cite{Mossel-Inverse},Proposition~3.9).}
\end{example}

We conclude the paper with an open question:

\begin{question}
Is this true that for small values of $\epsilon$, the GSWF on
three alternatives in which the choice functions are threshold
functions with expectations $\epsilon,1/2,$ and $1-\epsilon$
(i.e., the GSWF presented in the example above) has the highest
probability of rational choice amongst all non-dictatorial GSWFs
satisfying the IIA condition which are $\epsilon$-far from not
having all orderings of the alternatives in their range?
\end{question}

\section{Acknowledgements}

We are grateful to Tomer Schlank for helping to prove
Lemma~\ref{Lemma0.0}. It is a pleasure to thank Gil Kalai for
raising the questions addressed in the paper and for numerous
fruitful discussions. Finally, we thank Orr Dunkelman and the
anonymous referee for valuable suggestions.


\begin{thebibliography}{99}

\bibitem{Arrow} K.J. Arrow, A Difficulty in the Concept of Social
Welfare, {\it Journal of Political Economy} \textbf{58}(4) (1950),
pp.~328–-346.

\bibitem{Beckner} W. Beckner, Inequalities in Fourier
Analysis, {\it Annals of Math.} \textbf{102} (1975), pp.~159--182.

\bibitem{Ben-Or} M. Ben-Or and N. Linial, Collective Coin Flipping, in
{\it Randomness and Computation} (S. Micali, ed.), Academic Press,
New York, 1990, pp.~91--115.

\bibitem{Bonamie} A. Bonamie, Etude des Coefficients Fourier
des Fonctiones de $L^p (G)$, {\it Ann. Inst. Fourier} \textbf{20}
(1970), pp.~335--402.

\bibitem{Condorcet} M. de Condorcet, An Essay on the Application
of Probability Theory to Plurality Decision Making, 1785.

\bibitem{Erdos-Ko-Rado} P. Erd\"{o}s, C. Ko, and R. Rado,
Intersection Theorems for Systems of Finite Sets, {\it Quart. J.
Math. Oxford (2)}, \textbf{12} (1961), pp.~313--320.

\bibitem{FKG} C.M. Fortuin, P.W. Kasteleyn, and J. Ginibre,
Correlation Inequalities on Some Partially Ordered Sets, {\it
Comm. Math. Phys.} \textbf{22} (1971), pp.~89--103.

\bibitem{Geneakoplos} J. Geneakoplos, Three Brief Proofs of
Arrow's Impossibility Theorem, Cowels Foundation Discussion Paper
number 1123R, Yale University, 1997. Available online at:
http://ideas.uqam.ca/ideas/data/Papers/cwlcwldpp1123R.html

\bibitem{Gehrlein} W.V. Gehrlein, Condorcet's Paradox and the
Condorcet Efficiency of Voting Rules, {\it Math. Japon.}
\textbf{45} (1997), pp.~173--199.

\bibitem{Hilton-Milner} A.J.W. Hilton and C.E. Milner, Some
Intersection Theorems for Systems of Finite Sets, {\it Quart. J.
Math. Oxford (2)}, \textbf{18} (1967), pp.~369--384.

\bibitem{Kalai-Choice} G. Kalai, A Fourier-theoretic Perspective
on the Condorcet Paradox and Arrow's Theorem, {\it Adv. in Appl.
Math.} \textbf{29} (2002), no. 3, pp.~412--426.

\bibitem{Kalai-Private} G. Kalai, private communication, 2007.

\bibitem{Keller-new} N. Keller, A Tight Stability Version of
Arrow's Theorem, preprint, 2009.

\bibitem{Upfal} M. Mitzenmacher and E. Upfal, Probability and
Computing: Randomized Algorithms and Probabilistic Analysis,
Cambridge University Press, 2005.

\bibitem{Mossel-Inverse} E. Mossel, R. O'Donnell, O. Regev, J.E.
Steif, and B. Sudakov, Non-Interactive Correlation Distillation,
Inhomogeneous Markov Chains, and the Reverse Bonami-Beckner
Inequality, {\it Israel J. Math.} \textbf{154} (2006),
pp.~299--336.

\bibitem{Mossel} E. Mossel, R. O'Donnel, and K. Oleszkiewicz,
Noise Stability of Functions with Low Influences: Invariance and
Optimality, {\it Annals of Math.}, to appear.

\bibitem{Mossel-new} E. Mossel, Gaussian Bounds for Noise
Correlation of Functions and Tight Analysis of Long Codes,
proceedings of FOCS 2008, pp.~156--165, IEEE, 2008.

\bibitem{Mossel-Arrow} E. Mossel, A Quantitative Arrow Theorem,
preprint, 2009. Available online at:
http://arxiv.org/abs/0903.2574

\end{thebibliography}
\end{document}